\documentclass[letterpaper, 10 pt, conference]{ieeeconf}  % Comment this line out if you need a4paper

\IEEEoverridecommandlockouts                              % This command is only needed if
\usepackage{amsfonts}
\usepackage{amsmath}
\usepackage{epstopdf}
\usepackage{graphicx}
\usepackage{booktabs}
\usepackage{xcolor}
\usepackage{subfigure}
\usepackage{threeparttable}
\usepackage{footnote}
\newcommand{\R}{\mathbb{R}}

\newcommand{\N}{\mathbb{N}}
\newcommand{\bmat}[1]{\begin{bmatrix}#1\end{bmatrix}}
\newcommand{\norm}[1]{\left\lVert{#1}\right\rVert}

\newcommand{\ip}[2]{\left\langle #1, #2 \right\rangle}

\newcommand{\mcl}[1]{\mathcal{ #1}}
\newcommand{\mbf}[1]{\mathbf{ #1}}

\newtheorem{thm}{Theorem}

\newtheorem{lem}[thm]{Lemma}

\begin{document}

\title{\LARGE \bf {Estimator-Based Output-Feedback Stabilization of Linear Multi-Delay Systems using SOS \thanks{This work was supported by the National Science Foundation under grants No. 1538374 and 1739990, National Natural Science Foundation of China No. 61825304, 61751309, 61673335 and China Scholarship Council No. 201808130194.
}
}
%Convexification of the Synthesis Problem for Infinite-Dimensional Systems with Application to MIMO Multi-Delay Systems and Implementation in SOS
}

\author{Shuangshuang Wu\thanks{S. Wu is with the Institute of Electrical Engineering, Yanshan University, Qinhuangdao, 066004, China. e-mail: {\tt \small ssw\underline{\hbox to 2mm{}}0538\underline{\hbox to 2mm{}}ysu@163.com } }, Matthew~M.~Peet%
\thanks{M. Peet is with the School for the Engineering of Matter, Transport and Energy, Arizona State University, Tempe, AZ, 85298 USA. e-mail: {\tt \small mpeet@asu.edu } }, Changchun Hua\thanks{C. Hua is with the Institute of Electrical Engineering, Yanshan University, Qinhuangdao, 066004, China. e-mail: {\tt \small cch@ysu.edu.cn } }
}

\maketitle

%%%%%%%%%%%%%%%%%%%%%%%%%%%%%%%%%%%%%%%%%%%%%%%%%%%%%%%%%%%%%%%%%%%%%%%%%%%%%%%%
\begin{abstract}
In this paper, we investigate the estimator-based output feedback control problem of multi-delay systems. This work is an extension of recently developed operator-value LMI framework for infinite-dimensional time-delay systems. Based on the optimal convex state feedback controller and generalized Luenberger observer synthesis conditions we already have, the estimator-based output feedback controller is designed to contain the estimates of both the present state and history of the state. An output feedback controller synthesis condition is proposed using SOS method, which is expressed in a set of LMI/SDP constraints.  The simulation examples are displayed to demonstrate the effectiveness and advantages of the proposed results.
\end{abstract}
%\begin{abstract}
%Optimal controller synthesis is a bilinear problem and hence difficult to solve in a computationally efficient manner. We are able to resolve this bilinearity for systems with delay by first convexifying the problem in infinite-dimensions - formulating the $H_\infty$ optimal state-feedback controller synthesis problem for distributed-parameter systems as a Linear Operator Inequality - a form of convex optimization with operator variables. Next, we use positive matrices to parameterize positive ``complete quadratic'' operators - allowing the controller synthesis problem to be solved using Semidefinite Programming (SDP). We then use the solution to this SDP to calculate the feedback gains and provide effective methods for real-time implementation. Finally, we use several test cases to verify that the resulting controllers are \textit{optimal} to several decimal places as measured by the minimal achievable closed-loop $H_\infty$ norm, and as compared against controllers designed using high-order Pad\'e approximations.
%\end{abstract}

\section{Introduction}
Time delay widely exists in natural and engineered systems, often as a source of instability. Many works have been done on the study and control of time-delay systems during the last decades \cite{b1,b2}, mainly focusing on stability analysis, such as \cite{b6} and \cite{b7}. Despite the considerable advances that have been made in the area of stability analysis, the problem of stabilization of time-delay systems has been relatively neglected~\cite{b2,b8}. The primary problem in feedback stabilization of time-delay systems is the bilinearity between the controller and the Lyapunov certificate of stability. This bilinearity implies that combining parameterized controllers with standard approaches to Lyapunov-Krasovskii functional construction will result in Bilinear Matrix Inequalities -- a problem for which no efficient optimization algorithms exist. Faced with this bilinearity, some papers use iterative methods to alternately optimize the Lyapunov functional and then the controller as in \cite{b9,b10}. However, this iterative approach is not guaranteed to converge. Recently, however, duality-based methods have been proposed within the SOS-based operator-theoretic framework -- resulting in an LMI-based solution to the problem of $H_\infty$-optimal full-state-feedback control of multi-delay systems~\cite{b11}. The primary disadvantage of the full-state feedback controllers proposed in~\cite{b11} is that they assume accurate knowledge of all states of the system and moreover knowledge of the history of these states. Specifically, the controllers have the form
\begin{equation}
u(t)=K_0x(t)+\sum_i^K K_{1i}x(t-\tau_i)+\sum_i^K\int_{-\tau_i}^0K_{2i}(s)x(t+s)ds \label{eqn:controller}
\end{equation}
where the $H_\infty$-optimal controller gains $K_0, K_{1i}, K_{2i}$ are polynomials chosen to minimize the closed-loop $L_2$-gain bound $\gamma_1:=\text{sup}_{\omega\in L_2} \frac{\parallel z \parallel_{L_2}}{\parallel\omega\parallel_{L_2}}$. This formulation specifically precludes output-feedback controllers of the form $u(t)=Ky(t)$ or even $u(t)=Kx(t)$. In most practical cases such detailed measurements are not available.

The question of how to use measured outputs to reconstruct the full state is that of estimator design and is itself an area of active study (e.g. the Smith predictor can be thought of as an estimator using delayed output signals \cite{b16}). The $H_\infty$-optimal estimator design problem for multi-delay systems was itself directly addressed in the SOS-operator framework in \cite{b12}, wherein the observer is a simulated PDE running parallel to the real system which corrects both the present states and the history of the states. This observer minimizes an $L_2$-gain bound on the effect of disturbances on a regulated error signal.

In this paper, we propose a framework for using controllers of the form in Eqn.~\eqref{eqn:controller} where the controller acts not on the full state, but the state estimate derived from a dynamic estimator constructed using the algorithm proposed in~\cite{b12}. Specifically, the closed-loop dynamics have the form

\vspace{-0.4cm}
{\small{
\begin{align}
&\dot{x}(t)=A_0x(t)+\sum_i A_{i}x_{i}(t-\tau_i)+B_1w(t)+B_2u(t)\notag \\
 &\dot{\hat{x}}(t)=A_0\hat{x}(t)+\sum_i A_{i}\hat{\phi}_{i}(t,-\tau_i)+L_1b_0(t)\notag\\
 &\quad +\sum_{i}L_{2i}b_i(t,-\tau_i)+\sum_{i}\int_{-\tau_i}^{0}L_{3i}(s)b_i(t,s)ds\notag\\
&\partial_t\hat{\phi}(t,s)=\partial_s\hat{\phi}(t,s)+L_4(s)b_0(t)+\sum_{j}L_{5ij}b_j(t,-\tau_j)\notag\\
&\quad +L_{6i}(s)b_i(t,s)+\sum_j\int_{-\tau_i}^0L_{7ij}(s,\theta)b_j(t,\theta)d\theta\notag\\
&\hat{\phi}_i(t,0)=\hat{x}(t) \quad b_i(t,s)=C_2\hat{\phi}_i(t,s)-y(t+s)\notag\\
 &b_0(t)= C_2\hat{x}(t)-y(t)\notag\\
 &u(t)=K_0\hat x(t)+\sum_i K_{1i}\hat x(t-\tau_i)+\sum_i\int_{-\tau_i}^0K_{2i}(s)\hat x(t+s)ds\notag\\
 &y(t)=C_2x(t)+D_2w(t)\notag\\
&z(t)=C_{10}x(t)+\sum_i C_{1i}x_i(t-\tau_i)+D_1w(t)\notag\\
&z_e(t)=C_{30}e(t)+\sum_i C_{3i}e_i(t,-\tau_i)+D_3w(t)\label{eqn:nominal}
\end{align}}}
where $x(t)\in \R^n$ is the state, $\hat{x}(t)\in \R^n$ is the estimate of state, $\hat{\phi}(t,s) \in \R^n$ is the estimate of history of state, $w\in L_2^r$ is an external disturbance input, $u(t)\in \R^m$ is the actuated input, $y(t)\in \R^{q}$ is the
measured output, $z(t)\in \R^{p}$ is the regulated output, $z_e(t)\in \R^{p_1}$ is the estimated error of regulated output (not need to be $z(t)$ defined above). The delays $\tau_i>0$ for $i\in[1,\ldots,K]$ are ordered by increasing magnitude and $A_0, A_{i}, B_1, B_2, C_{10}, C_{1i}, C_2, C_{30}, C_{3i}, D_1, D_2, D_3$ are constant matrices with appropriate dimensions. We assume $x(0)=\hat{x}(t)=0$ for all $s\in[-\tau_K,0]$. The gains $K_0, K_{1i}, K_{2i}$ come from~\cite{b11} and the gains $L_0$, $L_{1i}$,  $L_{2i}$,  $L_{3i}$,  $L_{4i}$,  $L_{5ij}$ come from~\cite{b12}. By exploiting the properties of the gains and examining the dynamics of the closed-loop system, we show that the resulting dynamics are stable and establish a bound on the $H_\infty$-gain of the resulting closed-loop system. We furthermore propose a scheme for real-time numerical implementation of the observer-based controller and use numerical simulation to show that the resulting closed-loop system achieves internal stabilization.
\vspace{-0.3cm}
\subsection{Notation}
Shorthand notation used throughout this paper includes the Hilbert spaces $L_2^m[X]:=L_2(X;\R^m)$ of square integrable functions from $X$ to $\R^m$ and $W_2^m[X]:=W^{1,2}(X;\R^m)=H^1(X;\R^m)=\{x:x,\dot{x}\in L_2^m[X]\}$. We use $L_2^m, W_2^m$ when domains are clear from context. We also use the extensions $L_2^{n\times m}[X]:=L_2(X;\R^{n\times m})$ and $W_2^{n\times m}[X]:=W^{1,2}(X;\R^{n\times m})$ for matrix-valued functions. $S^n\subset \R^n\times n$ denotes the symmetric matrices. An operator $\mathcal{P}:Z\rightarrow Z$ is positive on a subset $X$ of Hilbert space $Z$ if $\left< x, \mathcal{P}x\right>\geq 0$ for all $x\in X$. $\mathcal{P}$ is coercive on $X$ if $\left< x, \mathcal{P}x\right>\geq\epsilon\lVert x \rVert^2_Z$ for some $\epsilon>0$ for all $x\in X$. Given an operator $\mathcal{P}:Z\rightarrow Z$ and a set $X\rightarrow Z$, we use the shorthand $\mathcal{P}(X)$ to denote the image of $\mathcal{P}$ on subset $X$. $I_n\in S^n$ denotes the identity matrix. $0_{n\times m}\in \R^{n\times m}$ is the matrix of zeros matrix with shorthand $0_n:=0_{n\times n}$. We will occasionally denote the intervals $T_i:=[-\tau_i,0]$. For a natural number, $K\in N$, we adopt the index shorthand notation which denotes $[K]={1,\cdots, K}$. The symmetric completion of a matrix is denoted $*^T$.

\section{Previous Work on State Estimation and State-Feedback Control of DPS}
In this section, we consider the a general class of distributed-parameter system (DPS) given as
\begin{align}
\dot{\mbf{x}}(t)&=\mcl{A}\mbf{x}(t)+\mcl{B}_1\omega(t)+\mcl{B}_2u(t)\quad \mbf x(0)=0\notag\\z(t)&=\mcl{C}_1\mbf{x}(t)+\mcl{D}_1\omega(t)\notag\\ \mbf{y}(t)&=\mcl{C}_2\mbf{x}(t)+\mcl{D}_2\omega(t) \label{eqn:DPS_FSC}
\end{align}
where $\mcl{A}:X \rightarrow Z$, $\mcl{B}_1:\R\rightarrow Z$, $\mcl{B}_2:U\rightarrow Z$, $\mcl{C}_1:X\rightarrow \R$, $\mcl{C}_2:X\rightarrow Y$, $\mcl{D}_1:\R\rightarrow \R$ and $\mcl{D}_2:\R\rightarrow Y$.
\vspace{-0.2cm}
\subsection{Full State feedback controller design}
\begin{thm} \cite{b11}
Suppose $\mcl{P}_1$ is a bounded, coercive linear operator $\mcl{P}_1:X\rightarrow X$ with $\mcl{P}_1(X)=X$ and which is self-adjoint with respect to the $Z$ inner product. Then $\mcl{P}_1^{-1}$ exists; is bounded; is self-adjoint; $\mcl{P}_1^{-1}:X\rightarrow X$; and $\mcl{P}_1^{-1}$ is coercive.
\end{thm}
\begin{thm} \cite{b11}
Suppose there exists a scalar $\epsilon_1>0$, an operator $\mcl{P}_1:Z\rightarrow Z$ which satisfies the conditions of Theorem 1, and an operator $\mcl{H}:X \rightarrow U$ such that

\vspace{-0.4cm}
{\small{
\begin{align}
&\left<\mcl{AP}_1\mbf{h},\mbf{h}\right>_Z+\left<\mbf{h},\mcl{AP}_1\mbf{h}\right>_Z+\left<\mcl{B}_2\mcl{H}\mbf{h},\mbf{h}\right>_Z+\left<\mbf{h},\mcl{B}_2\mcl{H}\mbf{h}\right>_Z\notag\\
&+\left<\mcl{B}_1\omega,\mbf{h}\right>_Z+\left<\mbf{h},\mcl{B}_1\omega\right>_Z-\gamma_1\lVert\omega\rVert^2-\gamma_1\lVert\upsilon\rVert^2+\upsilon^T(\mcl{C}_1\mcl{P}\mbf{h})\notag\\
&+(\mcl{C}_1\mcl{P}\mbf{h})^T\upsilon+\upsilon^T(\mcl{D}_2\mcl{H}\mbf{h})+(\mcl{D}_2\mcl{H}\mbf{h})^T\upsilon+\upsilon^T(\mcl{D}_1\omega)\notag\\
&+(D_1\omega)^T\upsilon\leq-\epsilon_1\lVert \mbf{h}\rVert^2
\end{align}}} for all $\mbf{h}\in X$, $\omega\in \R^r$ and $\upsilon\in \R^p$. Then if $\omega$ and $z$ satisfy Eqn.~\eqref{eqn:DPS_FSC} and  $u(t)=\mcl{K}\mbf{x}(t)$ where $\mcl{K}=\mcl{HP}_1^{-1}$ we have $\lVert z \rVert_{L_2}\leq\gamma_1\lVert\omega\rVert_{L_2}$.
\end{thm}
\vspace{-0.2cm}
\subsection{Estimator design}
In \cite{b12}, a $H_\infty$ optimal estimator based on the traditional Luenberger structure is given for Eqn.~\eqref{eqn:DPS_FSC}, which can correct both the present states and history of the states and give a real-time estimate of the history of states. This estimator has the following dynamics
\begin{align}
\dot{\hat{\mbf{x}}}(t)=\mcl{A}\hat{\mbf{x}}(t)+\mcl{L}(\mcl{C}_2\hat{\mbf{x}}(t)-\mbf{\mathrm{y}}(t)) \label{eqn:DPS_ED1}
\end{align}
for a given operator $\mcl{L}:Y\rightarrow Z$. By defining $\mbf{e}(t)=\hat{\mbf{x}}(t)-\mbf{x}(t)$, one obtains the error dynamics as
\begin{align}
\mbf{\dot{e}}(t)&=(\mcl{A}+\mcl{L}\mcl{C}_2)\mbf{e}(t)-(\mcl{B}_1+\mcl{L}\mcl{D}_2)\omega(t)\notag\\
z_e(t)&=\mcl{C}_3\mbf{e}(t)+\mcl{D}_3\omega(t)\quad \mbf{e}(0)=0 \label{eqn:DPS_ED2}
\end{align}
where $\mcl{C}_3:X\rightarrow \R$ and $\mcl{D}_3:\R\rightarrow \R$.

\begin{thm}\cite{b12}
Suppose there exist a scalar $\epsilon_2>0$ and bounded linear operators $\mcl{P}_2:Z\rightarrow Z$ and $\mcl{Z}: Y\rightarrow Z$ such that $\mcl{P}_2$ is coercive and

\vspace{-0.3cm}
{\small{
\begin{align}
&\left<(\mcl{P}_2\mcl{A}+\mcl{ZC}_2)\mbf{e},\mbf{e}\right>_Z+\left<\mbf{e},(\mcl{P}_2\mcl{A}+\mcl{ZC}_2)\mbf{e}\right>_Z\notag\\&-\left<\mbf{e},(\mcl{P}_2\mcl{B}_1+\mcl{ZD}_2)\omega\right>_Z-\left<(\mcl{P}_2\mcl{B}_1+\mcl{ZD}_2)\omega,\mbf{e}\right>_Z\notag\\&-\gamma_2\lVert\omega\rVert^2\notag-\gamma_2\lVert\upsilon_e\rVert^2+\left<\upsilon_e,\mcl{C}_3\mbf{e}\right>+\left<\mcl{C}_3\mbf{e},\upsilon_e\right>
\notag\\&+\left<\upsilon_e,\mcl{D}_3\omega\right>+\left<\mcl{D}_3\omega,\upsilon_e\right>\leq-\epsilon_2\lVert \mbf{e}\rVert^2\quad
\end{align}}}
for all $\mbf{e}\in X$, $\omega\in \R^r$ and $\upsilon_e\in \R^{p_1}$. Then $\mcl{P}_2^{-1}$ is a bounded linear operator and for $\mcl{L}=\mcl{P}_2^{-1}\mcl{Z}$, the solution of Eqn.~\eqref{eqn:DPS_ED2} satisfies
$\lVert z_e\rVert_{L_2}\leq\gamma_2\lVert \omega\rVert_{L_2}$.
\end{thm}
\section{Main results}
In this section, we give conditions under which the dynamics of the estimator-based controller is stable and give an expression for the $L_2$-gain of the closed-loop system. The conditions are given in abstract form. Later, in Theorem~\ref{thm:main}, we will given LMI-based sufficient conditions under which the conditions of Theorem~\ref{thm:abstract} is satisfied.
\vspace{-0.1cm}
\subsection{Estimator-Based Control for DPS}
Combining Eqn.~\eqref{eqn:DPS_FSC}, Eqn.~\eqref{eqn:DPS_ED1}, and Eqn.~\eqref{eqn:DPS_ED2} with $u(t)=\mcl{K}\hat{\mbf{x}}$, the closed-loop DPS dynamics are given as follows
\begin{align}
\dot{\mbf{x}}(t)&=(\mcl{A}+\mcl{B}_2\mcl{K})\mbf{x}(t)+\mcl{B}_1\omega(t)+\mcl{B}_2\mcl{K}\mbf{e}(t)\quad\notag\\
\mbf{\dot{e}}(t)&=(\mcl{A}+\mcl{L}\mcl{C}_2)\mbf{e}(t)-(\mcl{B}_1+\mcl{L}\mcl{D}_2)\omega(t)\notag\\
z(t)&=\mcl{C}_1\mbf{x}(t)+\mcl{D}_1\omega(t)\notag\\
\mbf y(t)&=\mcl{C}_2\mbf{x}(t)+\mcl{D}_2\omega(t)\label{eqn:DPS}\\
z_e(t)&=\mcl{C}_3\mbf{e}(t)+\mathcal{D}_3\omega(t)\notag
\end{align} where $\mcl{K}:Z\rightarrow U$ and $\mcl{L}:Y\rightarrow Z$. We assume $\mbf{x}(0)=\mbf e(0)=0$.

\begin{thm}\label{thm:abstract}
Suppose there exist positive scalars $\epsilon_1, \epsilon_2$, operators $\mcl H:Z  \rightarrow U$ and $\mathcal{P}_1:Z\rightarrow Z$ which satisfy the conditions of Theorem 1 with $\gamma_1$, and operators $\mathcal{P}_2:Z\rightarrow Z$,  and $\mcl{Z}:Y\rightarrow Z$ which satisfy Theorems 2 and 3 with $\gamma_2$. Then if there exists positive scalar $r$ such that
\begin{align}\ip{\bmat{\mbf h\\\mbf e}}{ \mcl{M}\bmat{\mbf h\\\mbf e}}\leq 0
\end{align} for all $\mbf h, \mbf e \in X$, where
\begin{align*}
 \mathcal{M}=\left[\begin{array}{ccc}
     -\epsilon_1I & \mathcal{B}_2\mathcal{H}\mathcal{P}_1^{-1}\\
     (\mathcal{B}_2\mathcal{H}\mathcal{P}_1^{-1})^T &-r\epsilon_2I\\
     \end{array}\right].
\end{align*}
Then for any $z(t)$, $z_e(t)$ and $w(t)$ which satisfy Eqn.~\eqref{eqn:DPS} with $\mathcal{K}=\mathcal{HP}^{-1}$ and $\mathcal{L}=\mathcal{P}^{-1}\mathcal{Z}$, we have $\lVert z \rVert_{L_2}\leq \sqrt{\gamma_1(\gamma_1+r\gamma_2)}\lVert\omega\rVert_{L_2}$ and $\lVert z_e\rVert_{L_2}\leq\gamma_2\lVert \omega\rVert_{L_2}$.
\end{thm}
\begin{proof} Suppose $z(t)$, $z_e(t)$, $\mbf y(t)$ $w(t)$, $\mbf e(t)$, $\mbf x(t)$ satisfy Eqn.~\eqref{eqn:DPS}. Since $z_e(t)$ is only affected by $\omega(t)$, we have by Theorem 3 that $\lVert z_e\rVert_{L_2}\leq\gamma_2\lVert \omega\rVert_{L_2}$. Define
\begin{align}
V(t)=V_1(t)+rV_2(t)
\end{align}
where $V_1(t)=\ip{\mbf x(t)}{\mcl P^{-1}\mbf x(t)}_Z$ and $V_2(t)=\ip{\mbf{e}(t)}{\mathcal{P}_2\mbf{e}(t)}_Z$. If we define expand $V_2(t)$ and apply Theorem 3, we have
\begin{align*}
\dot{V}_2(t)&-\gamma_2\lVert\omega(t)\rVert^2\leq-\epsilon_2\lVert \mbf{e}(t)\rVert^2.  \end{align*}
If we define $\mbf h(t)=\mcl P_1^{-1}\mbf x(t)\in X$ and differentiate $V_1(t)$, we have
\begin{align*}
&\dot{V}_1(t)=\left<\mathcal{AP}_1\mbf{h}(t),\mbf{h}(t)\right>_Z+\left<\mbf{h}(t),\mathcal{AP}_1\mbf{h}(t)\right>_Z\\
&+\left<\mathcal{B}_2\mathcal{H}\mbf{h}(t),\mbf{h}(t)\right>_Z+\left<\mbf{h}(t),\mathcal{B}_2\mathcal{H}\mbf{h}(t)\right>_Z\\
&+\left<\mathcal{B}_2\mathcal{H}\mcl{P}^{-1}_1\mbf{e}(t),\mbf{h}(t)\right>_Z+\left<\mbf{h}(t),\mathcal{B}_2\mathcal{H}\mcl{P}^{-1}_1\mbf{e}(t)\right>_Z\\
&\quad+\left<\mathcal{B}_1\omega(t),\mbf{h}(t)\right>_Z+\left<\mbf{h}(t),\mathcal{B}_1\omega(t)\right>_Z
.\end{align*}

Applying Theorem 2, if we define $\upsilon(t)=\frac{1}{\gamma_2}z(t)$, one gets
\begin{align*}
&\dot{V}_1(t)-\gamma_1\lVert\omega(t)\rVert^2+\gamma_1\lVert\upsilon(t)\rVert^2\\
& \leq-\epsilon_1\lVert \mbf{h}(t)\rVert^2+\left<\mathcal{B}_2\mathcal{H}\mcl{P}^{-1}_1\mbf{e}(t),\mbf{h}(t)\right>_Z\\
&+\left<\mbf{h}(t),\mathcal{B}_2\mathcal{H}\mcl{P}^{-1}_1\mbf{e}(t)\right>_Z\notag
.\end{align*}

Combining the results above, we have
\begin{align}
\dot{V}(t)&-(\gamma_1+r\gamma_2)\lVert\omega(t)\rVert^2+\gamma_1\lVert\upsilon(t)\rVert^2\notag\\
&\leq-\epsilon_1\lVert \mbf{h}(t)\rVert^2-r\epsilon_2\lVert \mbf{e}(t)\rVert^2\notag\\
&+\left<\mathcal{B}_2\mathcal{H}\mcl{P}^{-1}_1\mbf{e}(t),\mbf{h}(t)\right>_Z+\left<\mbf{h}(t),\mathcal{B}_2\mathcal{H}\mcl{P}^{-1}_1\mbf{e}(t)\right>_Z\notag\\
&=\ip{\bmat{\mbf h(t)\\\mbf e(t)}}{ \mcl{M}\bmat{\mbf h(t)\\\mbf e(t)}}\notag
.\end{align}
Then if there exist a positive scalar $r$ such that Eqn. (9) is satisfied, it follows
\begin{align*}
\dot{V}(t)-(\gamma_1+r\gamma_2)\lVert\omega(t)\rVert^2+\gamma_1\lVert\upsilon(t)\rVert^2&\leq0
.\end{align*}

Integrating in time and using $V(0)=0$, we have
\begin{align*}
\lVert z \rVert_{L_2}&\leq\sqrt{\gamma_1(\gamma_1+r\gamma_2)}\lVert \omega\rVert_{L_2}
.\end{align*}
The proof is completed.
\end{proof}
\vspace{-0.3cm}
\subsection{Expressing Multi-delay system into DPS}
In this section, we apply Theorem 4 to the case of multi-delay systems. Specifically, we consider solutions to the system of equations given by Eqn.~\eqref{eqn:nominal}.

Firstly, considering $e(t)=\hat{x}(t)-x(t)$, we write Eqn.~\eqref{eqn:nominal} into the form in Eqn.~\eqref{eqn:DPS_FSC}. Following the mathematical formalism developed in \cite{b2}, define the inner-product space $Z_{m,n,K}$:=$\{\R^m\times$ $L_2^n[-\tau_1,0]\times\cdots\times L_2^n[-\tau_K,0]\}$ and for $\{x,\phi_1,\cdots,\phi_K\}\in Z_{m,n,K}$, we use the following notation
\begin{align*}
\bmat{
x \\
\phi_i}:=\{x,\phi_1,\cdots,\phi_K\}
\end{align*}
and we define the inner product on $Z_{m,n,K}$ as
\begin{align*}
\left<
\bmat{
y\\
\psi_i
},\bmat{
x\\
\phi_i
}
\right>_{Z_{m,n,K}}=\tau_Ky^Tx+\sum_{i=1}^K\int_{-\tau_i}^0\psi_i(s)^T\phi_i(s)ds.
\end{align*}
We simplify the notation $Z_{m,n,k}$ when $m=n$ as $Z_{n,k}$.\\
Then the state-space for system (8) is defined as
\begin{align*}
X:=\left\{
\bmat{
x\\
\phi_i}
\in\mathcal{Z}_{n,K}:\begin{array}{c}
\phi_i\in W_2^n[-\tau_i,0]\text{ and }\phi_i(0)=x  \\
\text{ for all }i\in[K]
\end{array}
\right\}.
\end{align*}
We now represent the infinitesimal generator, $\mcl{A}:X\rightarrow Z_{n,K}$ of Eqn.~\eqref{eqn:DPS} as
\begin{align*}
\mathcal{A}\bmat{
x\\\phi_i
}(s)&=\bmat{
A_0x(t)+\sum A_i\phi_i(-\tau_i)\\
\dot{\phi_i}(s)
}.
\end{align*}
Furthermore, $\mcl{B}_1:\R^r\rightarrow Z_{n, K}$, $\mcl{B}_2:\R^m\rightarrow Z_{n, K}$, $\mcl{C}_1: Z_{n, K}\rightarrow \R^p$, $\mcl{C}_2: X\rightarrow Z_{q, K}$, $\mcl{C}_3: Z_{n, K}\rightarrow \R^{p_1}$, $\mcl{D}_1: \R^r\rightarrow \R^{p}$, $\mcl{D}_3: \R^r\rightarrow Z^{p_1}$  are defined as
\begin{align}
\mathcal{B}_1\omega(t)&:=\bmat{
B_1\omega(t)\\
0
}\quad
\mathcal{B}_2u(t):=\bmat{
B_2u(t)\\\
0
}\notag\\
\mathcal{C}_j\bmat{
x\\\phi_i}(s)&:=C_{j0}x(t)+\sum_i C_{ji}\phi_i(-\tau_i)\quad j=1, 3\notag\\
\mathcal{C}_2\bmat{
x\\\phi_i}(s)&:=
\bmat{C_2x(t)  \\
C_2\phi_i(s)}\\
\mathcal{D}_j\omega(t)&:=D_j\omega(t)\quad j=1,3.\notag
\end{align}

Here we assume $\mcl{D}_2=0$. Note for any solution $x(t)$ of Eqn.~\eqref{eqn:nominal}, using the above notation
\begin{align*}
(\mbf{x}(t))(s)=\left[\begin{array}{c}
x(t)\\
x(t+s)
\end{array}\right],
\end{align*}
then $\mbf{x}(t)$ satisfies Eqn.~\eqref{eqn:DPS}. The converse statement is also true. The same is true for $\mbf e(t)$, $\mbf y(t)$.
\vspace{-0.2cm}
\subsection{The operators framework}
A class of operators $\mathcal{P}_{\{P,Q_{i}, S_{i},R_{ij}\}}:Z_{m,n,K}\rightarrow Z_{m,n,K}$ is introduced which is parameterized by matrix $P$ and matrix-valued functions $Q_i\in W_2^{m\times n}[-\tau_i,0]$, $S_i\in W_2^{n\times n}[-\tau_i,0]$, $R_{ij}\in W_2^{n\times n}[-\tau_i,0]\times[-\tau_j,0]$ as
\begin{align}
&\left(\mathcal{P}_{\{P, Q_i ,S_i, R_{ij}\}}\bmat{
x  \\
\phi_i
}\right)(s):=\label{eqn:PQRS}\\&\bmat{
Px+\sum_i^K\int_{-\tau_i}^0Q_{i}(s)\phi_i(s)ds\\
\tau_KQ^T_i(s)+\tau_KS_i(s)\phi_i(s)+\sum_i^K\int_{-\tau_j}^0R_{ij}(s,\theta)\phi_j(\theta)d\theta
}\notag.
\end{align}
\begin{lem}\cite{b11}
	Suppose that $P\in \R^{n\times n}$, $S_{i}\in W_2^{n\times n}[T_i]$, $R_{ij}\in W_2^{n\times n}[T_i\times T_j]$ satisfying $S_i(s)=S_i^T(s)$, $R_{ij}(s, \theta)=R_{ji}^T(\theta, s)$, $P=\tau_KQ_i^T(0)$ and $Q_j(s)=R_{ij}(0,s)$ for all $i, j \in[K]$. Moreover suppose $\mathcal{P}_{\{P, Q_i ,S_i, R_{ij}\}}$ is coercive on $Z_{n,K}$. Then $\mathcal{P}_{\{P, Q_i ,S_i, R_{ij}\}}$ is a self-adjoint bounded linear operator with respect to the inner product defined on $Z_{n,K}$; $\mathcal{P}:X\rightarrow X$; and $\mathcal{P}_{\{P, Q_i ,S_i, R_{ij}\}}(X)=X$.
\end{lem}

Now let us turn to the operators used in Theorem 4. We define $\mcl P_1:=\mcl P_{\{P_1, Q_{1i}, S_{1i}, R_{1ij}\}}$ and $\mcl P_2:=\mcl P_{\{P_2, Q_{2i}, S_{2i}, R_{2ij}\}}$ and we parameterize the decision variable $\mathcal{H}:Z_{n,k}\rightarrow \R^m$ using matrices $H_0, H_{1i}$ and functions $H_{2i}$ as
{\small{\begin{align}
&\mathcal{H}\bmat{
y \\
y_i
}(s)=\left[H_0y+\sum_iH_{1i}y_i(-\tau_i)+\sum_i\int_{-\tau_i}^0H_{2i}(s)y_i(s)ds\right].\label{defn:H}
\end{align}}}
Similarly, the decision variable $\mcl Z$ is parameterized as
{\small{
\begin{align}
\mathcal{Z}\bmat{
y  \\
y_i}(s)&=\bmat{Z_1y_0+\sum_i Z_{2i}y_i(-\tau_i)+\sum_i\int_{-\tau_i}^0Z_{3i}(s)y_i(s)ds\notag\\
\tau_Kz_i(s)}\\
z_i(s)&=Z_{4i}(s)y_0+\sum_j Z_{5ij}(s)y_j(-\tau_j)+Z_{6i}(s)y_i(s)\notag\\&+\sum_j \int_{-\tau_j}^0 Z_{7ij}(s,\theta)y_j(\theta)d\theta.\label{defn:Z}
\end{align}}}
In~\cite{b12}, it was shown that for $\mcl Z$ as parameterized above, if $\mcl L=\mcl P^{-1}_1\mcl Z$, then the error injection operator $\mathcal{L}:Z_{q,k}\rightarrow Z_{n,k}$ corresponds to the estimator structure defined in Eqn.~\eqref{eqn:nominal}. The same is true for $\mcl K=\mcl H\mcl P_2^{-1}$.

To simplify presentation, we do not present the LMI constraints on the coefficients of $\{P, Q_i ,S_i, R_{ij}\}$ which ensure $\mathcal{P}_{\{P, Q_i ,S_i, R_{ij}\}} \ge 0$. Rather, we simply represent these constraints using the following notation.
\begin{align*}
&\Xi_{d,m,n,K}:=\\
&\left\{\{P,Q_i,R_{ij},S_i\}:\substack{ \{P,Q_i,R_{ij},S_i\} \text{ satisfy the conditions of Corollary 4} \\\text{in~\cite{b12}} \hspace{-.5mm} } \right\}.
\end{align*}

By Theorem 8 in \cite{b11}, if $\{P-\epsilon I,Q_i,R_{ij},S_i-\epsilon I \} \in \Xi_{d,m,n,K}$, then $\mathcal{P}_{\{P, Q_i ,S_i, R_{ij}\}}$ is coercive and has an inverse of the form $\hat{\mcl P}:=\hat{\mcl P}_{\{\hat{P}, \frac{1}{\tau_K}\hat{Q}_i ,\frac{1}{\tau_K^2}\hat{S}_i, \frac{1}{\tau_K}\hat{R}_{ij}\}}$. In this paper, we do not explicitly represent the map to $\{\hat{P}, \hat{Q}_i ,\hat{S}_i, \hat{R}_{ij}\}$, but rather combine it into a single map from $\{P, Q_i ,S_i, R_{ij}\}$ and $\{Z_1, Z_{2i}, Z_{3i}, Z_{4i}, Z_{5ij}, Z_{6i}, Z_{7ij}\}$ (resp. $\{H_1,H_{2i}, H_{3i}\}$) to $\{L_1, L_{2i}, L_{3i}, L_{4i}, L_{5ij}, L_{6i}, L_{7ij}\}$ (resp. $\{K_0, K_{1i}, K_{2i}\}$) which we then denote using the following.

\textbf{Definition of $\mcl L_o$:}
\vspace{-0.3cm}
{\small{
\begin{align*}
\{L_1, L_{2i},\cdots, L_{7ij}\} = \mcl L_{o}(\{P, Q_i ,S_i, R_{ij}\},\{Z_1, Z_2 ,\cdots, Z_{7ij}\})
\end{align*}}}
to indicate that if $\{\hat{P}, \hat{Q}_i ,\hat{S}_i, \hat{R}_{ij}\}$ are as defined in Theorem~8 in~\cite{b11}, then  $\{L_1, L_{2i}, \cdots, L_{7ij}\}$, $\{\hat{P}, \hat{Q}_i ,\hat{S}_i, \hat{R}_{ij}\}$, and $\{Z_1, Z_2 ,\cdots, Z_{7ij}\}$ satisfy Lemma 7 in~\cite{b12}.

%For the following analysis use, we quote the result of the controller reconstruction \cite{b11}
%
%\vspace{-0.3cm}
%{\small{
%\begin{align}
%u(t)&=\mcl K\mbf x(t)\\
%&=K_0x(t)+\sum_i K_{1i}x(t-\tau_i)+\sum_i\int_{-\tau_i}^0K_{2i}(s)x(t+s)ds\notag
%\end{align}}}

\textbf{Definition of $\mcl L_c$:}
Likewise, we say
\[
\{K_0, K_{1i}, K_{2i}\} = \mcl L_{c}(\{P, Q_i , S_i, R_{ij}\},\{H_0, H_{1i}, H_{2i}\})
\]
to indicate that if $\{\hat{P}, \hat{Q}_i ,\hat{S}_i, \hat{R}_{ij}\}$ are as defined in Theorem~8 in~\cite{b11}, then $\{K_0, K_{1i}, K_{2i}\}$, $\{\hat{P}, \hat{Q}_i ,\hat{S}_i, \hat{R}_{ij}\}$, and $\{H_0, H_{1i}, H_{2i}\}$ satisfy Lemma 9 in \cite{b11}.

%To unify the following results, one Corollary is given here.
%
%\begin{cor} \cite{b12}
%	Suppose there exist $d\in N$, constant $\epsilon>0$, matrix $P\in \R^{m\times m}$, polynomials $Q_i, S_i, R_{ij}$ for $i,j\in[K]$ such that
%	\begin{align}
%	\mathcal{L}_1(P, Q_i, S_i, R_{ij})\in\Xi_{d,m,nK},
%	\end{align}
%where the notation of $\mathcal{L}_1(P, Q_i, S_i, R_{ij})$ is in the Appendix, then $\left<\mbf{x}, \mathcal{P}_{\{P. Q_i, S_i, R_{ij}\}}\mbf{x}\right>_{Z_{m, n, nK}}\geq\epsilon\lVert\mbf{x} \rVert^2$ for all $\mbf{x}\in Z_{m, n, nK}$.
%\end{cor}
\vspace{-0.1cm}
\subsection{Theorem 4 applied to Multi-delay systems}
In this section, we formulate the conditions of Theorem 4 into multi-delay systems as a linear operator inequality where all operators are the form of Eqn.~\eqref{eqn:PQRS}.
\begin{thm}\label{thm:main}
	Suppose there exist $d\in \N$, positive scalars $\epsilon, \epsilon_1, \epsilon_2$, $\gamma_1, \gamma_2$, $\{P_1, Q_{1i}, S_{1i}, R_{1ij}\}$ satisfying Lemma 5, matrices $ P_2\in \R^{n\times n}$, polynomials $S_{2i}, Q_{2i}\in W_2^{n\times n}[T_i]$, $R_{2ij}\in W_2^{n\times n}[T_i\times T_j]$, matrices $H_0, H_{1i}\in \R^{p\times n}$, polynomial $H_{2i}\in W_2^{p\times n}[T_i]$, matrices $Z_1, Z_{2i}\in \R^{n\times q}$, polynomials $Z_{3i}$, $Z_{4i}$, $Z_{5ij}$, $Z_{6i}\in W_2^{n\times q}[T_i]$ and $Z_{7ij}\in W_2^{n\times q}[T_i\times T_j]$ for all $i$, $j\in[K]$ such that
	\begin{align*}
\{P_1-\epsilon I,Q_{1i}, S_{1i},R_{1ij}\}&\in \Xi_{d,n,n,K}\\
     -\{E_1+\epsilon_1\hat{I}_1, F_{1i}, N_{1i}+\epsilon_1I, G_{1ij}\}&\in \Xi_{d,m_0,n,K}\\
       \{P_2-\epsilon I,Q_{2i}, S_{2i},R_{2ij}\}&\in \Xi_{d,n,n,K}\\
	-\{E_2+\epsilon_2\hat{I}_2, F_{2i}, N_{2i}+\epsilon_2I, G_{2ij}\}&\in \Xi_{d,m_1,n,K}
    \end{align*}
    where
	\begin{align*}
&\{E_1, F_{1i}, {H}_{1i}, G_{1ij}\} \\
&\qquad = \mcl L_{1}(\{P_1,Q_{1i}, S_{1i},R_{1ij}\}, \{H_0, H_{1i},H_{2i}\})\\
&\{E_2, F_{2i}, {H}_{2i}, G_{2ij}\} \\
&\qquad= \mcl L_{2}(\{P_1,Q_{2i}, S_{2i},R_{2ij}\}, \{Z_0, Z_{1i},Z_{2i},\cdots,\})
    \end{align*}
and $m_0=p+r+n(K+1)$, $m_1=p+r+n(K+1)$, $\mcl L_1$ and $\mcl L_2$ are as defined in Appendix, $\hat{I}_1=\text{diag}(0_{r+p}, I_n, 0_{nK})$ and $\hat{I}_2=\text{diag}(0_{r+p_{1}}, I_n, 0_{nK})$.

Let {\small{\[
\{L_1, L_{2i},\cdots, L_{7ij}\} = \mcl L_{o}(\{P, Q_i ,S_i, R_{ij}\},\{Z_1, Z_2 ,\cdots, Z_{7ij}\})\]}}
 and {\small{\[
\{K_0, K_{1i}, K_{2i}\} = \mcl L_{c}(\{P, Q_i, S_i, R_{ij}\},\{H_0, H_{1i}, H_{2i}\}).
\]}}
Now further suppose that $r>0$ and
    \begin{align}
    -\{E_3+\epsilon_3\hat{I}_3, F_{3i}, N_{3i}, 0\}&\in \Xi_{d,n(K+2),2n,K}
	\end{align}
where{\small{
\begin{align*}
  &E_3=\bmat{-\frac{\epsilon_1}{\tau_K}I& B_2K_0       & B_2K_{11}      & \hdots&B_2K_{1k}\\
           *^T         & -r\frac{\epsilon_2}{\tau_K}I&0             &\hdots  &0\\
           *^T         &*^T           &0 &\hdots  &0\\
            \vdots     &\vdots        &\vdots        &\ddots &\vdots\\
            *^T         &*^T          &*^T            &*^T   &0}\\
            &F_{3i}=\bmat{K^T_{2i}(s)B^T_2& 0& 0&\cdots&0\\ 0& 0&0&\cdots&0}^T\quad \\
            &N_{3i}=\bmat{-r\frac{\epsilon_2}{\tau_K}I&0\\0&-\epsilon_1I}.
\end{align*}}} and $\hat{I}_3=\text{diag}(I_{n}, 0_n, 0_{nK})$.
Then if $w$, $z$ and $z_e$ satisfy Eqn.~\eqref{eqn:nominal} for some $\mbf x$ and $\hat{\mbf x}$, we have $\lVert z \rVert_{L_2}\leq\sqrt{\gamma_1(\gamma_1+r\gamma_2)}\lVert \omega\rVert_{L_2}$ and $\lVert z_e\rVert_{L_2}\leq\gamma_2\lVert \omega\rVert_{L_2}$.
\end{thm}
\begin{proof}
Let $\mcl{A}, \mcl{B}_1, \mcl{B}_2, \mcl{C}_1, \mcl{D}_1, \mcl{C}_2, \mcl{D}_2$ be as defined in Eqn. (11). Now define $\mcl{L}$ as
 {\small{
\begin{align}
&\mcl{L}\bmat{
y_0  \\
y_i}(s)=\bmat{
L_1y_0+\sum_i L_{2i}y_i(-\tau_i)+\sum_i\int_{-\tau_i}^0L_{3i}(s)y_i(s)ds\\
l_i(s)}\notag
\\
&l_i(s)=L_{4i}(s)y_0+\sum_j L_{5ij}(s)y_j(-\tau_j)+L_{6i}(s)y_i(s)\notag\\&\qquad+\sum_j \int_{-\tau_j}^0 L_{7ij}(s,\theta)y_j(\theta)d\theta.
\end{align}}}
and $\mcl{K}$ as
{\small{
\begin{align}
u(t)&=\mcl K\mbf x(t)\\
&=K_0x(t)+\sum_i K_{1i}x(t-\tau_i)+\sum_i\int_{-\tau_i}^0K_{2i}(s)x(t+s)ds\notag.
\end{align}}}

Since $\{P_1-\epsilon I,Q_{1i}, S_{1i}-\epsilon I,R_{1ij}\}\in \Xi_{d,n,n,K}$ and $\{P_2-\epsilon I,Q_{2i}-\epsilon I, S_{2i},R_{2ij}\}\in \Xi_{d,n,n,K}$, $\mcl P_1:=\mcl P_{\{P_1, Q_{1i}, S_{1i}, R_{1ij}\}}$ and $\mcl P_2:=\mcl P_{\{P_2, Q_{2i}, S_{2i}, R_{2ij}\}}$ are coercive. Let $\mcl Z$ be as defined in~\eqref{defn:Z} and $\mcl H$ be as defined in~\eqref{defn:H}.
Now by Theorem 5 and Lemma 10 in~\cite{b11}, $\mcl K= \mcl H \mcl P_1^{-1}$ and by Theorem 5 and Lemma 9 in~\cite{b12}, $\mcl L=\mcl P_1^{-1}\mcl Z$.

Next, if we define
\begin{align*}
 \mathcal{M}=\left[\begin{array}{ccc}
     -\epsilon_1I & \mathcal{B}_2\mathcal{H}\mathcal{P}_1^{-1}\\
     (\mathcal{B}_2\mathcal{H}\mathcal{P}_1^{-1})^* &-r\epsilon_2I\\
     \end{array}\right]=\bmat{-\epsilon_1I & \mathcal{B}_2\mathcal{K} \\(\mathcal{B}_2\mathcal{K})^* &-r\epsilon_2I}
\end{align*}
and for $\mbf h, \mbf e \in X$, we expand the expression
\begin{align*}
&\ip{\bmat{\mbf h\\\mbf e}}{ \mcl{M}\bmat{\mbf h\\\mbf e}}\\
&=\ip{\mathcal{B}_2\mathcal{K}\mbf{e}} {\mbf{h}}+\ip{\mbf{h}} {\mathcal{B}_2\mathcal{K}\mbf{e}}_Z -\epsilon_1 \norm{\mbf h}^2 -r \epsilon_2 \norm{\mbf e}^2
\end{align*}
\vspace{-0.3cm}
Now let
\begin{align*}
\mbf h(s) = \bmat{h_1\\h_{2i}(s)}, \mbf e(s) = \bmat{e_1\\e_{2i}(s)}\end{align*}
We have
\begin{align*}
&\ip{\bmat{\mbf h\\\mbf e}}{ \mcl{M}\bmat{\mbf h\\\mbf e}}\\
&=2\tau_Kh_1^T (B_2K_0e_1 +\sum_iB_2 K_{1i}e_{2i}(-\tau_i)\\
&+\sum_i\int_{-\tau_{K}}^0B_2K_{2i}e_{2i}(s)ds)-\epsilon_1h_1^Th_1-r\epsilon_2e_1^Te_1\\
&-\epsilon_1\sum_i\int_{-\tau_{K}}^0h_{2i}^T(s)h_{2i}(s)ds-r\epsilon_2\sum_i\int_{-\tau_{K}}^0e_{2i}^T(s)e_{2i}(s)ds.
\end{align*}
If we define $f_1=[h_1^T, e^T_1, e^T_{2i}(-\tau_1),\cdots,e^T_2(-\tau_K)^T]^T$ and $f_{2i}(s)=[e_{2i}^T(s), h_{2i}^T(s)]^T$, then  we obtain
\begin{align*}
&\ip{\bmat{\mbf h\\\mbf e}}{ \mcl{M}\bmat{\mbf h\\\mbf e}}\\
&=\ip{\bmat{f_1\\f_2i(s)}} {\mcl{P}_{\{E_3, F_{3i}, N_{3i}, 0\}}\bmat{f_1\\f_{2i}(s)}}_{Z_{n(K+2), 2n, K}}\label{eqn:quadratic}.
\end{align*}
Since $-\{E_3+\epsilon_3\hat{I}_3, F_{3i}, N_{3i}, 0\}\in \Xi_{d,n(K+2),2n,K}$, we conclude that $\mcl M \le 0$ and hence all the conditions of Theorem~\ref{thm:abstract} are satisfied.
Finally, suppose that $y(t)$, $z(t)$, $z_e(t)$, and $x(t)$ satisfy Eqn.~\eqref{eqn:nominal}.
If we define $\mbf y=y$, and
{\small{
\[
(\mbf{x}(t))(s)=\left[\begin{array}{c}
x(t) \\
x(t+s)
\end{array}\right],\; (\mbf e(t))(s)=\bmat{\hat x(t)-x(t)\\\hat{x}(t+s)-x(t+s)},
\]}}
then $\omega(t)$ $y(t)$, $z(t)$, $z_e(t)$, $\mbf e(t)$ and $\mbf x(t)$ satisfy~\eqref{eqn:DPS} and hence by Theorem~\ref{thm:abstract}, we have that $\norm{z}\le \sqrt{\gamma_1 (\gamma_1+r \gamma_2)}\norm{\omega}$ and $\norm{z_e}_{L_2}\le \gamma_2 \norm{\omega}_{L_2}$.
\end{proof}

Theorem 7 provides a method for using LMIs to construct estimator-based output feedback controllers for systems with multiple delays, including a bound on the closed loop $L_2$-gain.

\section{Numerical implementation, testing, validation}
The algorithms described in this paper have been implemented in Matlab within the DelayTOOLs framework, which is based on SOSTOOLS and the pvar framework. All the tools needed are available online for validation or download on Code Ocean [3].

For simulation, a fixed-step forward-difference-based discretization method is used, with a different set of states representing each delay channel. In the simulation results given below, 20 spatial discretization points are used for each delay channel.

\subsection{Example 1}
In this example, we consider the unstable system modified from the result in \cite{b14} which is in the form of Eqn.~\eqref{eqn:nominal} with
\begin{align*}
&A_0=\left[\begin{matrix}
0 & 0 \\
0 & 1
\end{matrix}\right]\quad A_1=\left[\begin{matrix}
-1 & -1 \\
0 & -0.9
\end{matrix}\right] \quad B_1=\left[\begin{matrix}
1 & 0 \\
0 & 1
\end{matrix}\right]\\&B_2=\left[\begin{matrix}
0 \\ 1
\end{matrix}\right]\quad C_{10}=\left[\begin{matrix}
1 & 0
\end{matrix}\right]\quad D_{1}=\left[\begin{matrix}
1 & 0
\end{matrix}\right]\\
&C_{30}=\left[\begin{matrix}
1.5 & 0.5
\end{matrix}\right]\quad D_{3}=\left[\begin{matrix}
1 & 0
\end{matrix}\right]
\quad C_2=\left[\begin{matrix}
1 & 0
\end{matrix}\right]
\end{align*}
and $\tau=0.99$.

\subsection{Example 2}
This example is given by modifying the result from \cite{b16} which is in the form of Eqn.~\eqref{eqn:nominal} with
\begin{align*}
&A_0=\left[\begin{matrix}
-10 & 10  \\
0 & 1
\end{matrix}\right]\quad A_1=\left[\begin{matrix}
1 & 1  \\
1 & 1
\end{matrix}\right]\quad B_1=\left[\begin{matrix}
1 & 0 \\
0 & 1
\end{matrix}\right]\\&B_2=\left[\begin{matrix}
1 \\
1
\end{matrix}\right]u(t)\quad C_{10}=\left[\begin{matrix}
1 & 0 \\
0 & 1
\end{matrix}\right]\quad C_{30}=\left[\begin{matrix}
1.2 & 0 \\
0 & 1.2
\end{matrix}\right]\\&C_{2}=\left[\begin{matrix}
0 & 10
\end{matrix}\right] \quad\text{and} \quad \tau=0.3.
\end{align*}
\subsection{Example 3}
This example considers the 2-delay case as a modified version of Example 1, which is in the form of Eqn.~\eqref{eqn:nominal} with
\begin{align*}
&A_0=\left[\begin{matrix}
0 & 0  \\
0 & 1
\end{matrix}\right]\quad A_i=\left[\begin{matrix}
-0.5 & -0.5  \\
0 & -0.45
\end{matrix}\right]\quad i=1, 2\\
&B_1=\left[\begin{matrix}
1 & 0 & 0\\
0 & 1 & 0
\end{matrix}\right]\quad B_2=\left[\begin{matrix}
1 \\
1
\end{matrix}\right]\quad C_{10}=\left[\begin{matrix}
1 & 0
\end{matrix}\right]\\&C_{30}=\left[\begin{matrix}
1.1 & 0.2
\end{matrix}\right]\quad C_2=\left[\begin{matrix}
0 & 1
\end{matrix}\right]
\end{align*}
and $\tau_1=0.5, \tau_2=1$.

\begin{figure*}[htbp]
\centering
\begin{minipage}[t]{0.31\linewidth}
\centering
\includegraphics[height=1.5in,width=2.4in]{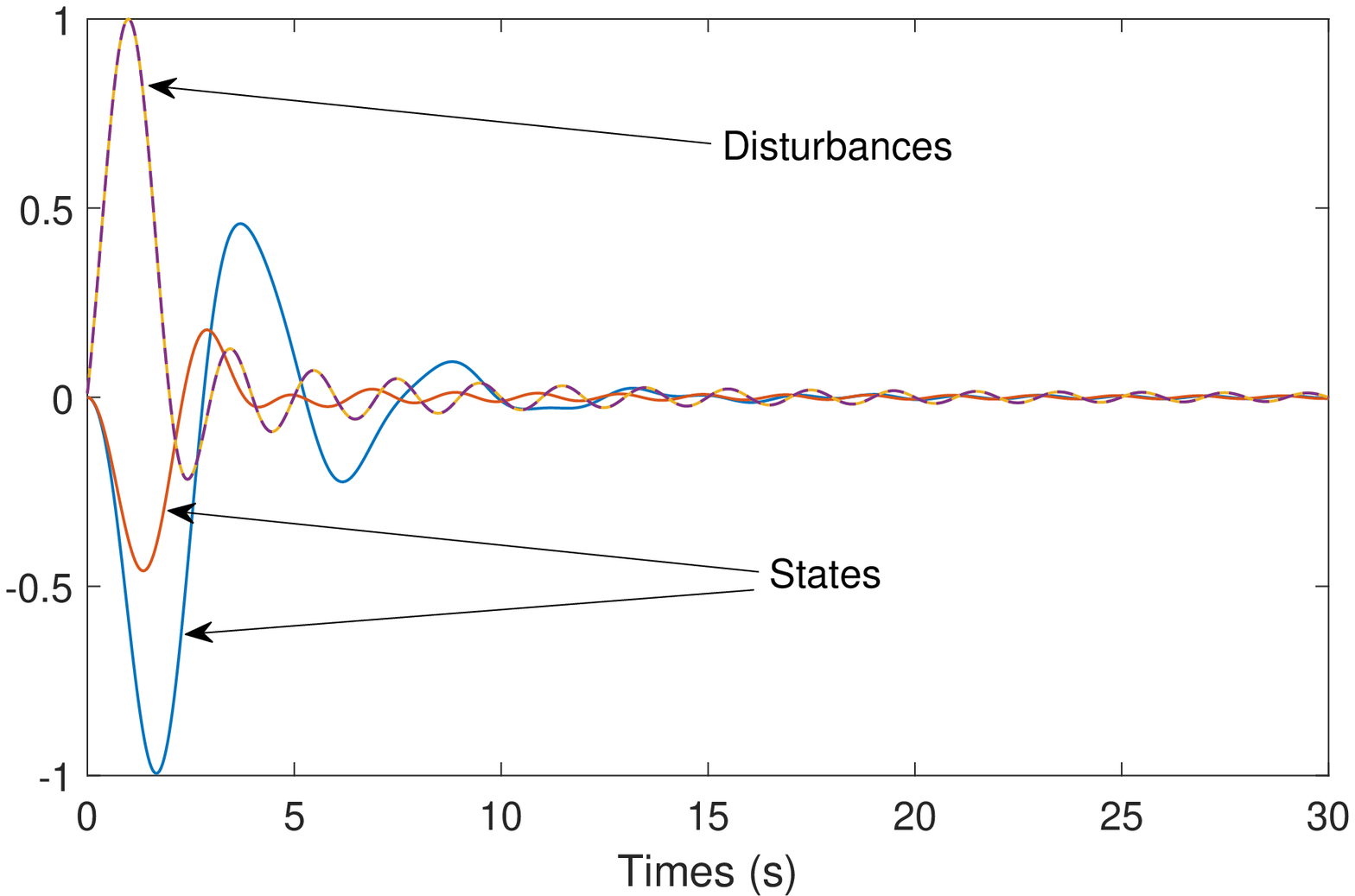}
\caption{Errors in the estimated state $e(t)$ in closed-loop response for a sinc disturbance for Example 1}
\end{minipage}
\quad
\begin{minipage}[t]{0.31\linewidth}
\centering
	 \includegraphics[height=1.5in,width=2.4in]{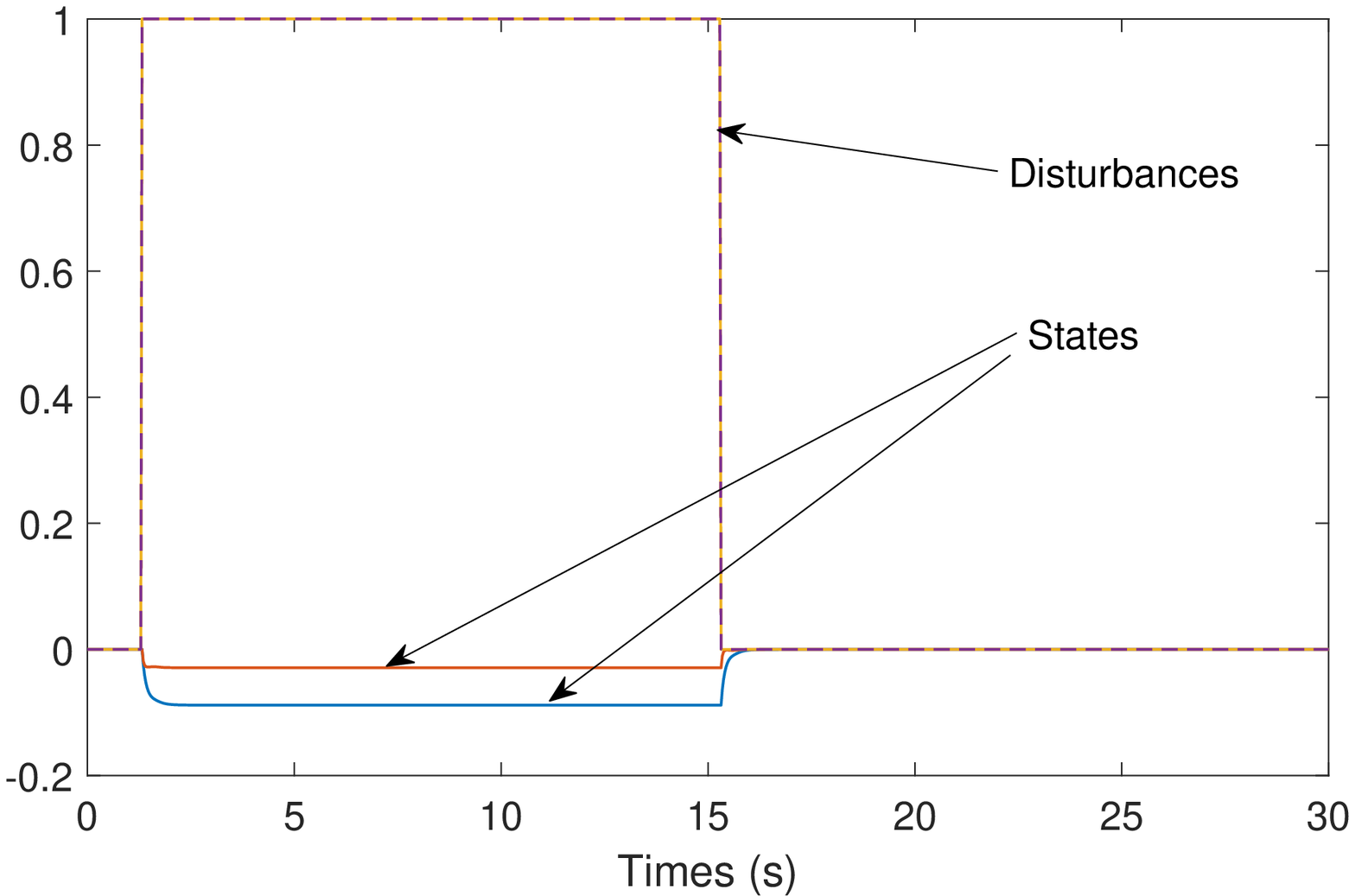}
	\caption{Errors in the estimated state $e(t)$ in closed-loop response for a step-like disturbance for Example 2}
\end{minipage}
\quad
\begin{minipage}[t]{0.31\linewidth}
\centering
	 \includegraphics[height=1.5in,width=2.4in]{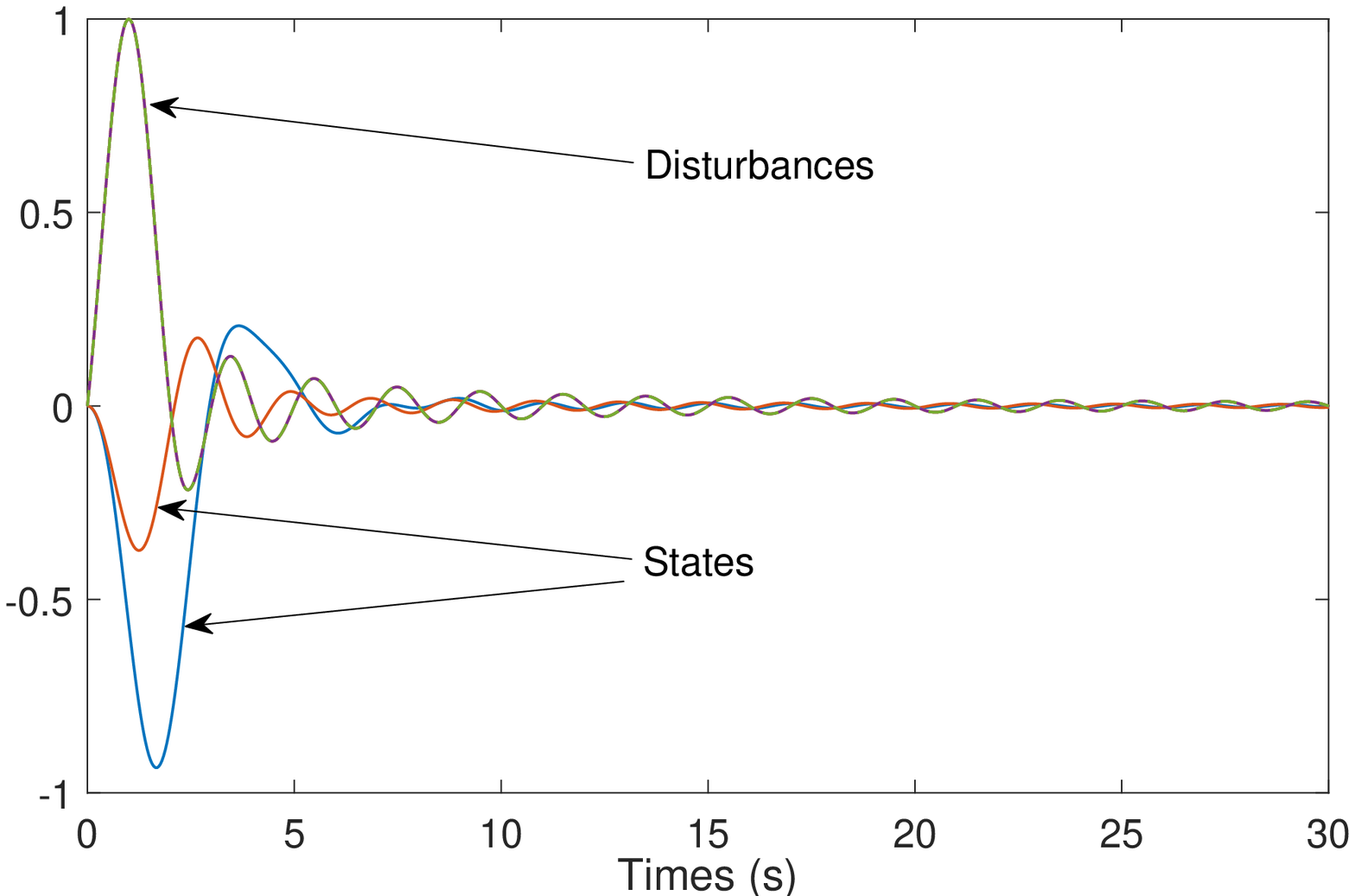}
	\caption{Errors in the estimated state $e(t)$ in closed-loop response for a sinc disturbance for Example 3}
\end{minipage}
\end{figure*}

\begin{figure*}[htbp]
\centering
\begin{minipage}[t]{0.31\linewidth}
\centering
    \includegraphics[height=1.5in,width=2.4in]{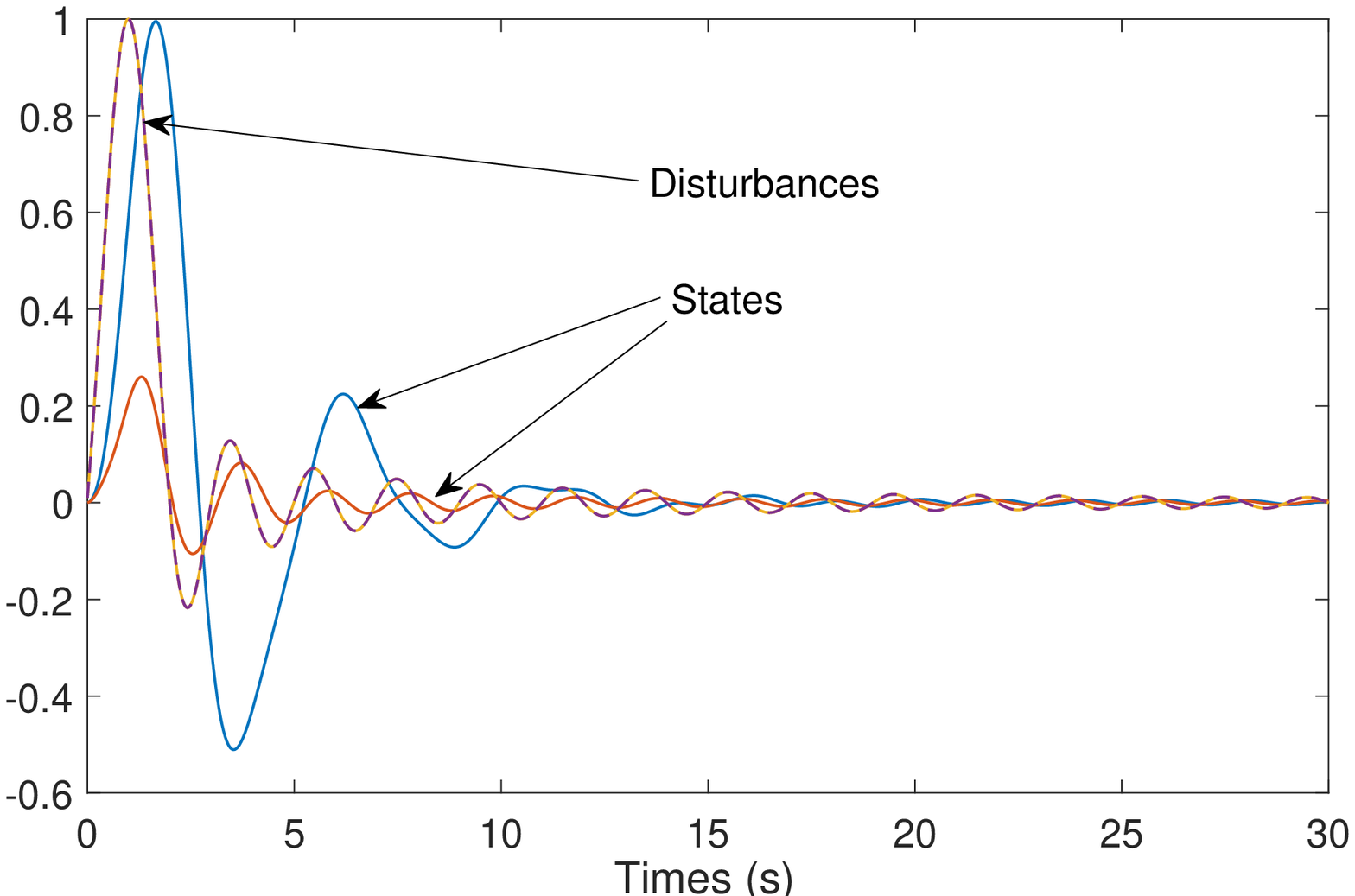}
	\caption{State trajectory $x(t)$ in closed-loop response for a sinc disturbance for Example 1}
\end{minipage}
\quad
\begin{minipage}[t]{0.31\linewidth}
\centering
	 \includegraphics[height=1.5in,width=2.4in]{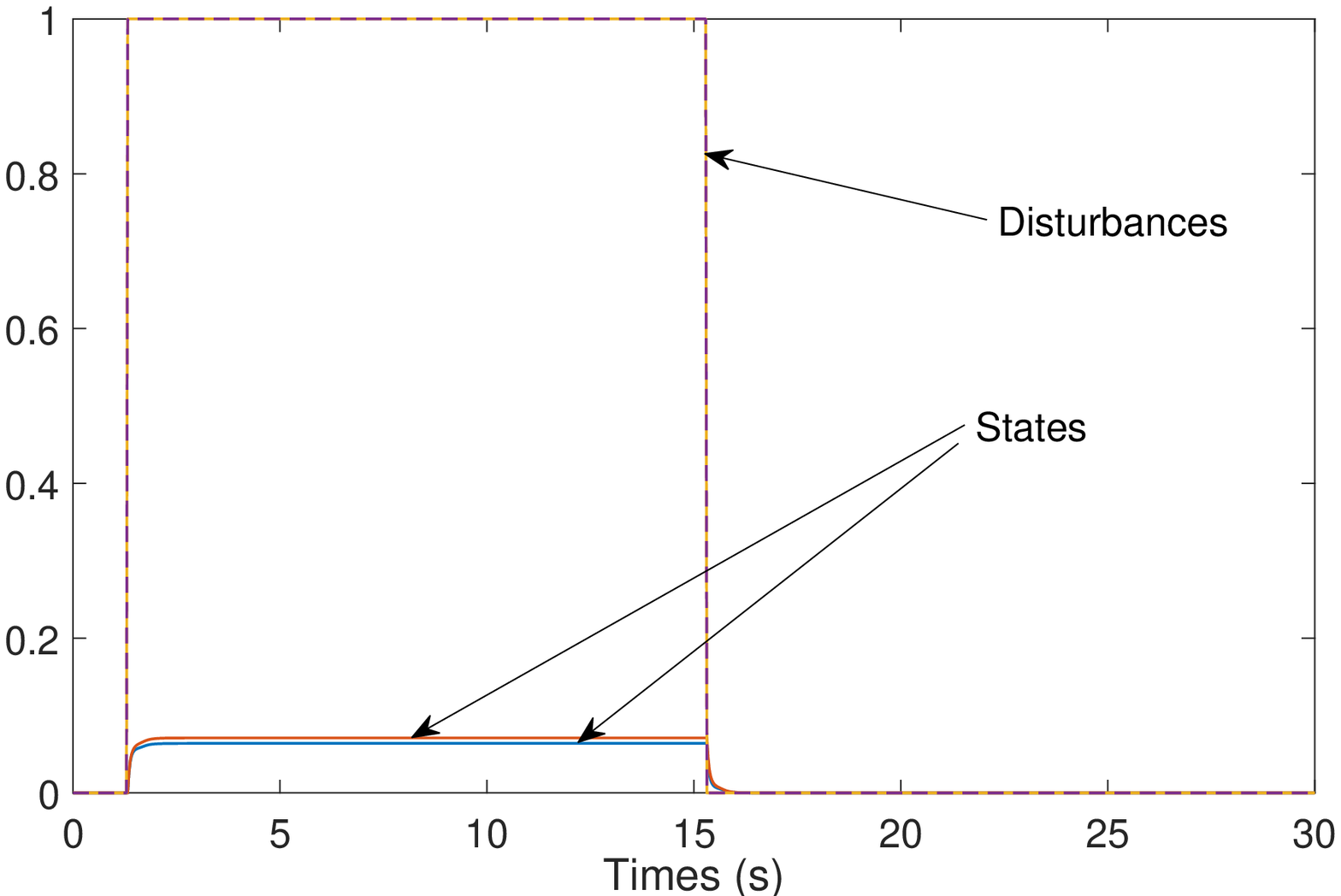}
	\caption{State trajectory $x(t)$ in closed-loop response for a step-like disturbance for Example 2}
\end{minipage}
\quad
\begin{minipage}[t]{0.31\linewidth}
\centering
	 \includegraphics[height=1.5in,width=2.4in]{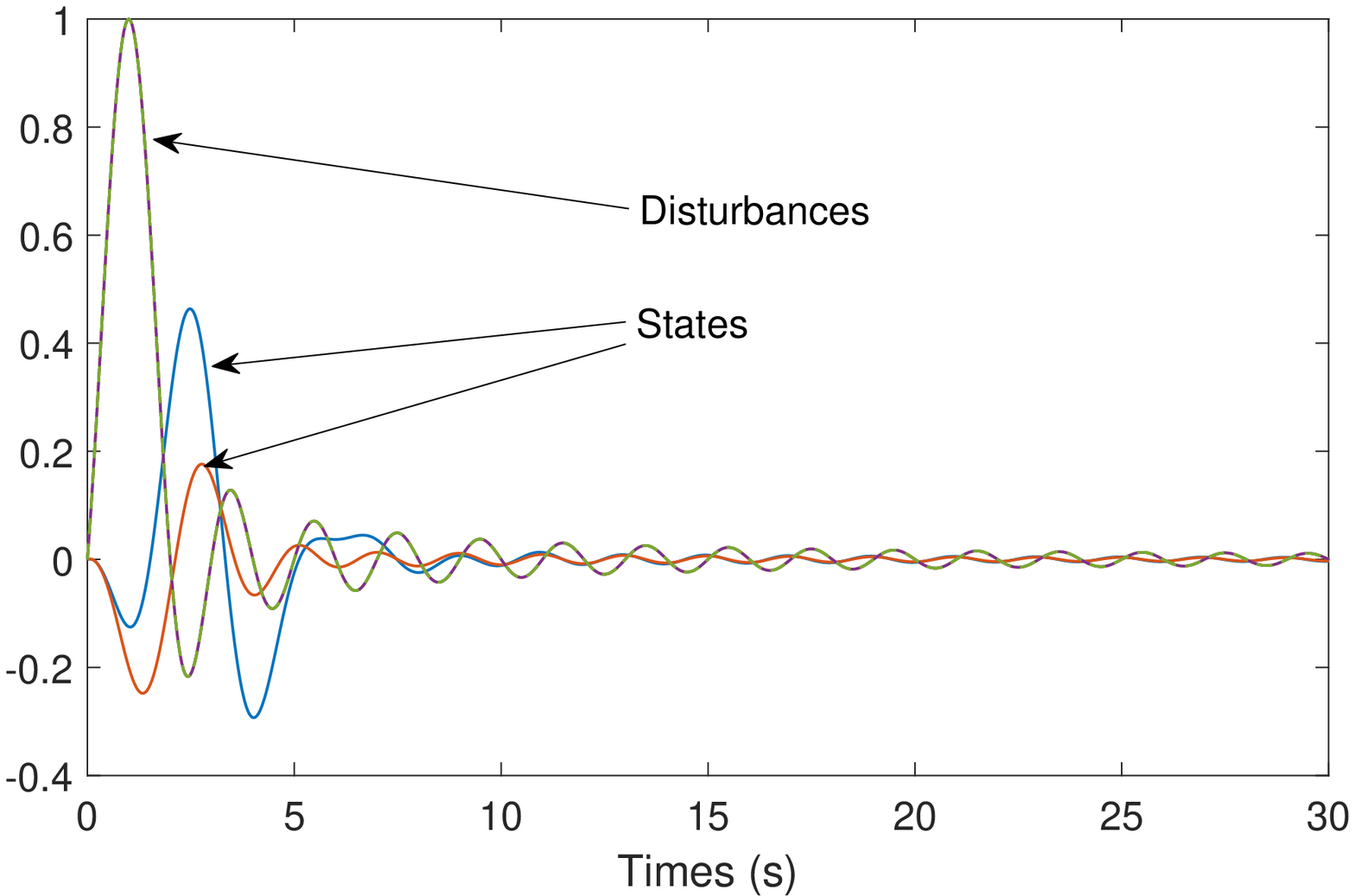}
	\caption{State trajectory $x(t)$ in closed-loop response for a sinc disturbance for Example 3}
\end{minipage}
\end{figure*}

These three numerical examples are used to validate and test the accuracy of the algorithm defined in Theorem~\ref{thm:main}. In each instance, we find a state feedback controller, an observer, and construct observer-based controller. In Table~\ref{tab:tab1}, we find the $\gamma_1,\gamma_2$ obtained from Theorem 7 as  compared to an H$_\infty$ optimal output feedback controller obtained by using a 10th order Pad$\acute{\text{e}}$ approximation of the delay terms in Table 1. We also give a lower bound on the real $L_2$ gain $\gamma_{\text{real}}$ by numerically simulating the effect of a disturbance $\omega(t)$ on the $L_2$-norm of the regulated output $z(t)$ and comparing to the $L_2$-norm of the input. The closed-loop dynamics are validated in Figs. 1-6 where we see the estimator-based controller is effective in stabilization of systems that are open-loop unstable.

\section{Conclusion}
In this paper, we have proposed a method for designing estimator-based output feedback controllers for systems with multiple delays. This approach combines an $H_\infty$-optimal estimator with an $H_\infty$-optimal full-state feedback controller and proves a bound on the $L_2$-norm of the resulting dynamics. These controllers are applicable to systems with multiple known delays and consider process noise, but not sensor noise. Furthermore, we have developed an efficient numerical implementation of the observer-based controller and have posted this implementation online. Numerical examples indicate that the $L_2$-gain of the resulting estimator-based controllers is relatively close to, but does not exactly achieve the minimum possible closed-loop $L_2$-gain as estimated using a Pad\'e approximation.

\vspace{-.3cm}
\begin{table}
\fontsize{4.5}{9}\selectfont
\begin{tabular}{c|c|c|c|c|c|c|c|c|c}
\hline
& \multicolumn{3}{|c|}{Example 1} & \multicolumn{3}{|c|}{Example 2} &
\multicolumn{3}{|c}{Example 3} \\ \cline{2-10}\cline{5-9}
& d=1 & d=2 & d=4 & d=1 & d=2 & d=4 & d=1 & d=2 & d=4 \\ \hline
$\gamma _{1\text{min} }$ & 1.9082 & 1.6829 & 1.6359 & 0.1067 & 0.1060 & 0.1059 &
1.049 & 0.9892 & 0.9596 \\ \hline
$\gamma _{2\text{min} }$ & 4.1485 & 4.1425 & 4.1425 & 0.1325 & 0.1325 & 0.1325 &
1.4862 & 1.4851 & 1.4851 \\ \hline
$\gamma _{\text{min} }$ & \multicolumn{3}{|c|}{3.0450} & \multicolumn{3}{|c|}{
0.1104} & \multicolumn{3}{|c}{1.3499} \\ \hline
$\gamma_{\text{real}}$ & \multicolumn{3}{|c|}{0.7893} &
\multicolumn{3}{|c|}{0.0738} & \multicolumn{3}{|c}{0.6080} \\ \hline
\end{tabular}
\caption{
In this table, $\gamma_{1\text{min}}$ and $\gamma_{2\text{min}}$ are the values $\gamma_1,\gamma_2$ in Theorem~\ref{thm:main}. $\gamma_{\text{min}}$ is the calculated minimized $L_2$ gain bound on the effect of the disturbance $\omega(t)$ on the regulated output $z(t)$ of the closed-loop system (2) under $H_\infty$ output feedback control using a 10th order Pad$\acute{\text{e}}$ approximation of the delay terms. $\gamma_{\text{real}}$ is the real $L_2$ gain on the effect of the disturbance $\omega(t)$ on the regulated output $z(t)$ of the numerical examples under the estimator-based controller we construct.}\label{tab:tab1}
\end{table}

%%%%%%%%%%%%%%%%%%%%%%%%%%%%%%%%%%%%%%%%%%%%%%%%%%%%%%%%%%%%%%%%%%%%%%%%%%%%%%%%
%\bibliographystyle{IEEEtran}
%\bibliography{peet_bib,delay,NSF_CAREER_bib2011,LMIs}

\section*{\large{Appendix}}
In this appendix, we define the mappings $\mcl L_1$ and $\mcl L_2$ as used in Theorem~\ref{thm:main}.

\textbf{Operator $\mcl L_1$:} We say
\begin{align*}
&\{E_1, F_{1i}, {N}_{1i}, G_{1ij}\}\\
&\qquad = \mcl L_{1}(\{P_1,Q_{1i}, S_{1i},R_{1ij}\}, \{H_0, H_{1i},H_{2i}\})
\end{align*}
if

{\small{
\begin{align*}
&E_1=\\
&\bmat{-\frac{\gamma_1}{\tau_K} I &\frac{1}{\tau_K}D_1&E_{11} &E_{121} &\hdots &E_{12K}\\
  *^T&-\frac{\gamma_1}{\tau_K} I &B_1^T&0&\hdots&0\\
  *^T&*^T&E_{10}+E_{10}^T&E_{131} &\hdots &E_{13K}\\
  *^T&*^T&*^T&-S_{11}(-\tau_1)&\hdots &0\\
  \vdots &\vdots&\vdots&\vdots&\ddots &\vdots\\
  *^T&*^T&*^T&*^T&\hdots &-S_{1k}(-\tau_K)}\\
  &F_{1i}(s)\\
&=\frac{1}{\tau_K}\cdot\\
&\bmat{C_{10}Q_{1i}(s)+\sum_j C_{1j} R_{1ji}(-\tau_j,s)\\0\\ \tau_K\left(A_0 Q_{1i}(s) +\dot Q_{1i}(s)+\sum_{j=1}^K A_j R_{1ji}(-\tau_j,s)+B_2H_{2i}(s)\right)  \\ 0 \\ \vdots \\0}\\
&N_{1i}(s)=\dot{S}_{1i}(s)\\
&G_{1ij}(s,\theta)\\
&=\frac{\partial}{\partial s}R_{1ij}(s,\theta)+\frac{\partial}{\partial \theta}R_{1ji}(s,\theta)^T, \quad i,j \in [K]
  \end{align*}}}
  where

{\small{
\begin{align*}
&E_{10}= A_0 P_1+ \sum_{i=1}^K \left(  \tau_K A_i Q_{1i}(-\tau_i)^T  + \frac{1}{2}S_{1i}(0)\right)+B_2H_0\\
&E_{11}=\frac{1}{\tau_K}C_{10}P_1+\sum_i C_{1i} Q_{1i}(-\tau_i)^T\\
&E_{12i}=C_{1i} S_{1i}(-\tau_i)\\
&E_{13i}=\tau_K A_iS_{1i}(-\tau_i)+B_2H_{1i}
\end{align*}}}

\textbf{Operator $\mcl L_2$:} We say

\begin{align*}
&\{E_2, F_{2i}, {H}_{2i}, G_{2ij}\} \\
&\qquad= \mcl L_{2}(\{P_1,Q_{2i}, S_{2i},R_{2ij}\}, \{Z_0, Z_{1i},Z_{2i},\cdots,\})
\end{align*}
if
{\small{
\begin{align*}
&E_2:=\mcl{L}_{5}(P_2, Q_{2i}, S_{2i}, Z_1, Z_{2i})\\
&=\bmat{
-\frac{\gamma_2}{\tau_K}I &D_3^T&-E_{20}^T&0&\hdots &0\\
  *^T&-\frac{\gamma_2}{\tau_K} I &\frac{C_{10}}{\tau_K}&\frac{C_{11}}{\tau_K}&\hdots&\frac{C_{1K}}{\tau_K}\\
  *^T&*^T&E_{210}&E_{211} &\hdots &E_{21K}\\
  *^T&*^T&*^T&-S_{21}(-\tau_1)&\hdots &0\\
  \vdots &\vdots&\vdots&\vdots&\ddots &\vdots\\
  *^T&*^T&*^T&*^T&\hdots &-S_{2k}(-\tau_K)}\\
 &F_{2i}(s):=\mcl{L}_{6}(Q_{2i}, R_{2ij}, Z_{3i}, Z_{4i}, Z_{5ij})\\
&=\bmat{-Q_{2i}^T(s)B_1&0&F_{20i}(s)&F_{21i}(s)&\hdots&F_{2Ki}(s)}^T
\\
&N_{2i}(s):=\mcl L_7(S_{2i}, Z_{6i})=\dot{S}_{2i}(s)+Z_{6i}(s)C_2+C_2^TZ^T_{6i}(s)\\
&G_{2ij}(s, \theta):=\mcl L_8(R_{2ij})=-\frac{\partial}{\partial s}R_{2ij}(s,\theta)-\frac{\partial}{\partial \theta}R_{2ij}(s,\theta)\\
&\quad+\tau_K(Z_{7ij}(s, \theta)C_2+C_2^TZ_{7ji}^T(\theta, s))
  \end{align*}}}
where
{\small{
\begin{align*}
&E_{20}:= P_2B_1\\
&E_{210}:=P_2A_0+A_0^TP_2+\sum_kQ_{2k}(0)+Q_{2k}^T(0)+S_{2k}(0)\\
&\quad+Z_1C_2+C_2^TZ_1^T\\
&E_{21i}=P_2A_i-Q_{2i}(-\tau_i)+Z_{2i}C_2\\
&F_{20i}(s)=A_0^TQ_{2i}(s)+\frac{1}{\tau_K}\sum_kR_{2ik}^T(s,0)-\dot{Q}_{2i}(s)\\
&\quad+Z_{4i}(s)C_2+C_2^TZ^T_{3,i}(s)\\
&F_{2ji}(s)=A_j^TQ_{2i}(s)+\frac{1}{\tau_K}R_{2ij}^T(s,-\tau_j)+Z_{5ij}(s)C_2.
\end{align*}}}

\end{document}